\documentclass[12pt,english,a4paper]{smfart}

\usepackage[T1]{fontenc}
\usepackage{lmodern}
\usepackage{smfthm}
\usepackage[headings]{fullpage}
\usepackage{amssymb}

\newcommand{\Id}{\operatorname{Id}}
\newcommand{\im}{\operatorname{im}}

\newcommand{\tor}{\operatorname{tor}}
\newcommand{\Qp}{\mathbf{Q}_p}

\newcommand{\Cp}{\mathbf{C}_p}
\newcommand{\Zp}{\mathbf{Z}_p}

\newcommand{\NN}{\mathbf{N}}
\newcommand{\QQ}{\mathbf{Q}}
\newcommand{\eps}{\varepsilon}
\newcommand{\OO}{\mathcal{O}}
\newcommand{\bigO}{\mathrm{O}}
\newcommand{\calE}{\mathcal{E}}
\newcommand{\calA}{\mathcal{A}}
\newcommand{\calF}{\mathcal{F}}
\newcommand{\calR}{\mathcal{R}}
\newcommand{\MM}{\mathfrak{m}}
\newcommand{\frB}{\mathfrak{B}}
\newcommand{\frX}{\mathfrak{X}}
\newcommand{\wideg}{\operatorname{wideg}}
\newcommand{\res}{\operatorname{res}}
\newcommand{\val}{\operatorname{val}}
\newcommand{\vp}{\val_p}
\newcommand{\Tr}{\operatorname{Tr}}
\newcommand{\Hom}{\operatorname{Hom}}
\newcommand{\card}{\operatorname{card}}
\newcommand{\dcroc}[1]{[\![ #1 ]\!]}
\newcommand{\LT}{\mathrm{LT}}
\newcommand{\Gm}{\mathbf{G}_\mathrm{m}}
\renewcommand{\geq}{\geqslant}
\renewcommand{\leq}{\leqslant} 
\renewcommand{\phi}{\varphi} 
\newcommand{\hotimes}{\widehat{\otimes}}
\newcommand{\fan}{\text{$F\!\mbox{-}\mathrm{la}$}}
\newcommand{\la}{\mathrm{la}}

\NumberTheoremsIn{subsection}

\author{Laurent Berger}
\address{UMPA de l'ENS de Lyon \\
UMR 5669 du CNRS}
\email{laurent.berger@ens-lyon.fr}
\urladdr{perso.ens-lyon.fr/laurent.berger/}

\title{Lubin-Tate generalizations of the $p$-adic Fourier transform}

\date{\today}

\begin{document}

\begin{abstract}
Fresnel and de Mathan proved that the $p$-adic Fourier transform is surjective. We reinterpret their result in terms of analytic boundaries, and extend it beyond the cyclotomic case. We also give some applications of their result to Schneider and Teitelbaum's $p$-adic Fourier theory, in particular to generalized Mahler expansions and to the geometry of the character variety. 
\end{abstract}

\subjclass{11S; 12J; 13J; 14G; 30G; 46S}

\keywords{$p$-adic Fourier transform; $p$-adic Fourier theory; Mahler expansion; Lubin-Tate group; character variety; analytic boundary}

\thanks{This research is partially supported by the ANR project ANR-19-CE40-0015 COLOSS}

\maketitle

\setcounter{tocdepth}{2}

\tableofcontents

\setlength{\baselineskip}{18pt}

\section*{Introduction}

\subsection*{The $p$-adic Fourier transform}
Let $\Cp$ be the completion of an algebraic closure of $\Qp$, and let $\Gamma = \{ \gamma \in \Cp$ such that $\gamma^{p^n}=1$ for some $n \geq 0\}$ be the set of roots of unity of $p$-power order. Let $c^0(\Gamma,\Cp)$ be the set of sequences $\{z_\gamma\}_{\gamma \in \Gamma}$ with $z_\gamma \in \Cp$ and such that $z_\gamma \to 0$ (namely: for every $\eps>0$, the set of $\gamma$ such that $|z_\gamma| \geq \eps$ is finite) and let $C^0(\Zp,\Cp)$ be the space of continuous functions $\Zp \to \Cp$. For every $\gamma \in \Gamma$, the function $a \mapsto \gamma^a$ belongs to $C^0(\Zp,\Cp)$. 

\begin{enonce*}{Definition}
The \emph{Fourier transform} of $z \in c^0(\Gamma,\Cp)$ is the function $\calF(z) : \Zp \to \Cp$ given by $a \mapsto \sum_{\gamma \in \Gamma} z_\gamma \cdot \gamma^a$. 
\end{enonce*}

Fresnel and de Mathan proved (see \cite{FdM74,FdM75,FdM78}) the following result.

\begin{enonce*}{Theorem}
The Fourier transform $\calF : c^0(\Gamma,\Cp) \to C^0(\Zp,\Cp)$ is surjective, and moreover $\calF : c^0(\Gamma,\Cp) / \ker \calF \to C^0(\Zp,\Cp)$ is an isometry. 
\end{enonce*}

Because of the appearance of roots of unity, the $p$-adic Fourier transform can be seen as a cyclotomic construction. In this paper, we generalize the definition of the Fourier transform as well as Fresnel and de Mathan's theorem beyond the cyclotomic case. We then give a mostly independent application of their theorem to Schneider and Teitelbaum's $p$-adic Fourier theory \cite{ST01}.

\subsection*{Analytic boundaries}
For the first generalization, consider the dual of the $p$-adic Fourier transform. The dual of $c^0(\Gamma,\Cp)$ is $\ell^\infty(\Gamma,\Cp)$, the set of bounded sequences. The dual of $C^0(\Zp,\Cp)$ is isomorphic to $\calE_{\Cp}^+ = \Cp \otimes_{\OO_{\Cp}} \OO_{\Cp} \dcroc{X}$ (via the Amice transform that sends a measure $\mu$ to $\calA_\mu(X) = \sum_{n \geq 0} \mu( a \mapsto \binom{a}{n}) \cdot X^n$).

The dual of the Fourier transform is hence a map $\calF' : \calE_{\Cp}^+ \to \ell^\infty(\Gamma,\Cp)$. It is easy to see that this map is given by $f(X) \mapsto \{ f(\gamma-1) \}_{\gamma \in \Gamma}$. Fresnel and de Mathan's theorem is then equivalent to the claim that $\calF'$ is an isometry on its image, namely that $\|f\|_D = \sup_{\gamma \in \Gamma} |f(\gamma-1)|$ where $D=\MM_{\Cp}$ is the $p$-adic open unit disk. 

\begin{enonce*}{Definition}
A subset $A =\{a_n\}_{n \geq 1} \subset D$ is an \emph{analytic boundary} if $|a_n| \to 1$ as $n \to +\infty$ and if for every $f \in \calE^+_{\Cp}$ we have $\| f \|_D = \|f\|_A := \sup_{n \geq 1} |f(a_n)|$.
\end{enonce*}

Fresnel and de Mathan's theorem is then equivalent to the claim that $\{ \gamma-1, \gamma \in \Gamma\}$ is an analytic boundary. We prove that the same holds if $A$ is the set of torsion points of a Lubin-Tate formal group attached to a finite extension of $\Qp$, and even more generally if $A$ is the set of iterated roots of a certain class of power series, that we call Lubin-Tate-like (LT-like) power series. Let $q$ be a power of $p$.

\begin{enonce*}{Definition}
An \emph{LT-like power series} (of Weierstrass deg $q$) is a power series $P(X) = \sum_{n \geq 1} p_n X^n \in \OO_{\Cp}\dcroc{X}$ with $0 < \vp(p_1) \leq 1$, $p_q \in \OO_{\Cp}^\times$ and $P(X) \equiv p_q X^q \bmod{p_1}$.
\end{enonce*}

If $P(X)$ is as above, let $\Lambda(P)  = \{ z \in D$ such that $P^{\circ n}(z)=0$ for some $n \geq 0\}$. The following result is theorem \ref{lambdas}.

\begin{enonce*}{Theorem A}
If $P$ is LT-like, then $\Lambda(P)$ is an analytic boundary.
\end{enonce*}

If $P(X) = (1+X)^p-1$, then $\Lambda(P) = \{ \gamma-1, \gamma \in \Gamma\}$, and theorem A implies the result of Fresnel and de Mathan. The proof of theorem A is very similar to Fresnel and de Mathan's proof of their result.

\subsection*{$p$-adic Fourier theory}
For the second generalization, let $F$ be a finite extension of $\Qp$ of degree $d$, with ring of integers $\OO_F$. Let $X_{\tor}$ denote the set of finite order characters $(\OO_F,+) \to (\Cp^\times,\times)$. Given $z \in c^0(X_{\tor},\Cp)$, its Fourier transform is the function $\calF(z) : \OO_F \to \Cp$ defined by $a \mapsto \sum_{g \in X_{\tor}} z_g \cdot g(a)$. It is easy to see (theorem \ref{fdmof}) that Fresnel and de Mathan's theorem implies that $\calF : c^0(X_{\tor},\Cp) \to C^0(\OO_F,\Cp)$ is surjective. We give an application of this observation to $p$-adic Fourier theory.

Let $e$ be the ramification index of $F$, let $\pi$ be a uniformizer of $\OO_F$, and let $q = \card \OO_F/\pi$. Let $\LT$ be the Lubin-Tate formal $\OO_F$-module attached to $\pi$,  let $X$ be a coordinate on $\LT$, and let $\log_{\LT}(X)$ be the logarithm of $\LT$. For $n \geq 0$, let $P_n(Y) \in F[Y]$ be the polynomial defined by $\exp(Y \cdot \log_{\LT}(X)) = \sum_{n \geq 0} P_n(Y) X^n$. 

When $F=\Qp$ and $\LT=\Gm$, we have $P_n(Y) = \binom{Y}{n}$. The family $\{ \binom{Y}{n} \}_{n \geq 0}$ forms a Mahler basis of $\Zp$. In addition, by a theorem of Amice \cite{A64}, every locally analytic function $\Zp \to \Cp$ can be written as $x \mapsto \sum_{n \geq 0} c_n \binom{x}{n}$ where $\{c_n\}_{n \geq 0}$ is a sequence of $\Cp$ such that there exists $r>1$ satisfying $|c_n| \cdot r^n \to 0$. 

In their work \cite{ST01} on $p$-adic Fourier theory, Schneider and Teitelbaum generalized this last result to $F \neq \Qp$. They proved the existence of an element $\Omega \in \OO_{\Cp}$, with $\vp(\Omega) = 1/(p-1)-1/e(q-1)$, such that $P_n(a \Omega) \in \OO_{\Cp}$ for all $a \in \OO_F$. The power series $G(X) = \exp(\Omega \cdot \log_{\LT}(X))-1$ therefore belongs to $\Hom_{\OO_{\Cp}}(\LT,\Gm)$. One of the main results of $p$-adic Fourier theory is the following (prop 4.5 and theo 4.7 of \cite{ST01}).

\begin{enonce*}{Theorem}
If $\{c_m\}_{m \geq 0}$ is a sequence of $\Cp$ such that there exists $r>1$ satisfying $|c_m| \cdot r^m \to 0$, then $a \mapsto \sum_{m \geq 0} c_m P_m(a \Omega)$ is a locally $F$-analytic function $\OO_F \to \Cp$.

Conversely, every locally $F$-analytic function $\OO_F \to \Cp$ has a unique such expansion.
\end{enonce*}

If we only ask that $c_m \to 0$, then $a \mapsto \sum_{m \geq 0} c_m P_m(a \Omega)$ is a continuous function $\OO_F \to \Cp$. We therefore get a map $c^0(\NN, \Cp) \to C^0(\OO_F,\Cp)$, whose image contains all locally $F$-analytic functions. If $F = \Qp$, this map is an isomorphism. In general, Fresnel and de Mathan's theorem and some computations in $p$-adic Fourier theory imply that the map is surjective, and noninjective if $F \neq \Qp$. Using the fact that every element of $C^0(\Zp,\Cp)$ can be written in one and only one way as $x \mapsto \sum_{n \geq 0} \lambda_n \binom{x}{n}$ where $\lambda \in c^0(\NN, \Cp)$, we reformulate this result using the following definition.

\begin{enonce*}{Definition}
The \emph{Peano map} $T : C^0(\Zp,\Cp) \to C^0(\OO_F,\Cp)$ is the map given by \[ T : \left[ x \mapsto \sum_{n \geq 0} \lambda_n \binom{x}{n} \right] \mapsto \left[ a \mapsto \sum_{n \geq 0} \lambda_n P_n( a \Omega) \right]. \]
\end{enonce*}

\begin{enonce*}{Theorem B}
The Peano map $T : C^0(\Zp,\Cp) \to C^0(\OO_F,\Cp)$ is surjective, and noninjective if $F \neq \Qp$.
\end{enonce*}

This is coro \ref{thb}. By Schneider and Teitelbaum's theorem recalled above, $T : C^{\la}(\Zp,\Cp) \to C^{\fan}(\OO_F,\Cp)$ is an isomorphism. So one can think of $T$ as some Peano-like map: a surjective noninjective limit of isomorphisms, from a $1$-dimensional object to a $d$-dimensional object.

\subsection*{The character variety}

The rigid analytic $p$-adic open unit disk $\frB$ is a parameter space for characters $(\Zp,+) \to (\Cp^\times,\times)$: if $K$ is a closed subfield of $\Cp$, a point $z \in \frB(K)$ corresponds to the character $\eta_z : a \mapsto (1+z)^a$ and all $K$-valued continuous characters are of this form. In particular, all continuous characters are locally analytic. 

If $F$ is a finite extension of $\Qp$ of degree $d$, then $\OO_F \simeq \Zp^d$ and $\frB^d$ is then a parameter space for characters $(\OO_F,+) \to (\Cp^\times,\times)$. Schneider and Teitelbaum have constructed in \cite{ST01} a $1$-dimensional rigid analytic group variety $\frX \subset \frB^d$ over $F$, called the character variety, whose closed points in an extension $K/F$ parameterize locally $F$-analytic characters $\OO_F \to K^\times$. They show that over $\Cp$, the variety $\frX$ becomes isomorphic to $\frB$.

Let $\OO^b_{\Cp}(\frB^d)$ denote the ring of bounded functions on $\frB^d$ defined over $\Cp$, and likewise for $\OO^b_{\Cp}(\frX)$. We have $\OO^b_{\Cp}(\frX) \simeq \calE^+_{\Cp}$ and $\OO^b_{\Cp}(\frB^d)$ is likewise isomorphic to the ring of bounded functions in $d$ variables. The restriction-to-$\frX$ map $\res_{\frX} : \OO^b_{\Cp}(\frB^d) \to \OO^b_{\Cp}(\frX)$ is injective by \cite{BSX}. By $p$-adic Fourier theory, $\OO^b_{\Cp}(\frX)$ is the dual of $C^0(\Zp,\Cp)$, $\OO^b_{\Cp}(\frB^d)$ is the dual of $C^0(\OO_F,\Cp)$, and $\res_{\frX}$ is the dual of the Peano map $T$. 

Theorem B now implies the following result (theorem \ref{resisom}).

\begin{enonce*}{Theorem C}
The map $\res_{\frX} : \OO^b_{\Cp}(\frB^d) \to \OO^b_{\Cp}(\frX)$ is an isometry on its image.
\end{enonce*}

In the isomorphism between $\frX$ and $\frB$, we have $\OO^b_{\Cp}(\frX) \simeq \calE^+_{\Cp}$, and the set $X_{\tor}$ of torsion characters $(\OO_F,+) \to (\Cp^\times,\times)$ corresponds to $\LT[\pi^\infty]$. Theorem A applied to $P(X)=[\pi](X)$ implies the following result (theorem \ref{suptors}).

\begin{enonce*}{Theorem D}
If $f \in \OO^b_{\Cp}(\frX)$, then $\| f \|_{\frX} = \sup_{\kappa \in X_{\tor}} |f(\kappa)|$. 
\end{enonce*}

Theorem A is proved in \S \ref{bdsec} and theorems B, C and D are proved in \S \ref{fousec}.

\section{Construction of analytic boundaries}
\label{bdsec}

The goal of this section is to state and prove theorem A.

\subsection{$p$-adic holomorphic functions and analytic boundaries}
\label{holrem}

We recall some standard facts about holomorphic functions on the $p$-adic open unit disk (for which see \cite{L62} or \cite{R00}), and define analytic boundaries. Let $D=\MM_{\Cp}$ be the $p$-adic open unit disk. Let $\calE_{\Cp}^+ = \Cp \otimes_{\OO_{\Cp}} \OO_{\Cp} \dcroc{X}$ be the ring of bounded holomorphic functions on $D$, and let $\calR_{\Cp}^+$ be the ring of holomorphic functions on $D$. If $f \in \calR_{\Cp}^+$ and $\mu > 0$, we let $V(f,\mu) = \inf_{n \geq 0} \vp(f_n) + \mu n$. If $\mu \in \QQ_{> 0}$, then $V(f,\mu) = \inf_{z \in D, \vp(z)=\mu} \vp(f(z))$. The function $\mu \mapsto V(f,\mu)$ is continuous, increasing and piecewise affine. We have $V(fg,\mu) = V(f,\mu)+V(g,\mu)$. If $f \in \calE_{\Cp}^+$, then $V(f,0)$ is also defined, and $V(f,0) = -\log_p \| f \|_D$. We say that $\mu>0$ is a critical valuation if there exists $i \neq j$ such that $V(f,\mu) = \vp(f_i) + \mu i = \vp(f_j) + \mu j$. Recall that $f$ has a zero of valuation $\mu$ if and only if $\mu$ is a critical valuation, and that the critical valuations of $f$, as well as the number of zeroes of $f$ having that valuation, can be read on the Newton polygon of $f$. 

Divisors are defined in \S 4 of \cite{L62}. In this paper, we only consider divisors that are an infinite formal product $\prod_{k \geq 1} D_k(X)$ where for each $k$, $D_k(X)$ is a polynomial such that $D_k(0)=1$ and all the roots of $D_k$ are of valuation $\mu_k$, where $\{ \mu_k \}_{k \geq 1}$ is a strictly decreasing sequence converging to $0$. We then have $V(D_k,\mu) = 0$ if $\mu \geq \mu_k$ and $V(D_k,\mu) = \deg D_k \cdot (\mu-\mu_k)$ if $\mu \leq \mu_k$.

\begin{prop}
\label{fdmzer}
Let $\prod_{k \geq 1} D_k(X)$ be a divisor and take $\eta > 0$. 

There exists $f(X) \in \calR^+$ such that $f(0)=1$, $f$ is divisible by $D_k$ for all $k \geq 1$, and for all $\mu>0$, we have $\sum_{k \geq 1} V(D_k,\mu) \geq V(f,\mu) \geq \sum_{k \geq 1} V(D_k,\mu) - \eta$.
\end{prop}

\begin{proof}
This is theorem 1 of \cite{FdM74}. See theorem 25.5 of \cite{E95} for a full proof, noting that $A_b(d(0,r^-))$ should be $A(d(0,r^-))$ in the statement of ibid.
\end{proof}

We now define analytic boundaries. Since $D$ is a separable topological space, there are plenty of countable sets $A =\{a_n\}_{n \geq 1} \subset D$ such that $\| f \|_D = \|f\|_A := \sup_{n \geq 1} |f(a_n)|$ for all $f \in \calE_{\Cp}^+$. We are interested in those sets $A$ such that $|a_n| \to 1$ as $n \to +\infty$.

\begin{defi}
\label{defprops}
We say that $A=\{a_n\}_{n \geq 1} \subset D$ is an analytic boundary if $|a_n| \to 1$ as $n \to + \infty$ and if for every $f \in \calE^+_{\Cp}$ we have $\| f \|_D = \|f\|_A$.
\end{defi}

\begin{lemm}
\label{remove}
If $A$ is an analytic boundary and $h \neq 0 \in \calE^+_{\Cp}$, then $A' = A \setminus \{ a \in A$ such that $h(a)=0\}$ is also an analytic boundary.
\end{lemm}

\begin{proof}
Since $A' \subset A$, it is clear that $|a'_n| \to 1$ as $n \to + \infty$. Moreover, $\| fh \|_A = \| fh \|_{A'}$. Hence if $f \in \calE^+_{\Cp}$, then $\| f \|_D \cdot \| h \|_D = \|f h\|_D = \| fh \|_A = \| fh \|_{A'} \leq \| f \|_{A'} \cdot \| h \|_D$.
\end{proof}

In particular, if $A$ is an analytic boundary, then $A_m = \{ a_n \}_{n \geq m}$ is an analytic boundary for all $m \geq 1$. Our definition of analytic boundary is therefore consistent with the definition of \emph{analytic boundary for $\calE^+_{\Cp}$} given in \S 2 of \cite{B10}, except that we require in addition that $|a_n| \to 1$ as $n \to + \infty$. The following result (theorem 8 of \cite{B10}) can be used to construct many examples of analytic boundaries.

\begin{theo}
\label{wellsep}
If $A \subset D \setminus \{0\}$ is such that $\sum_{n \geq 1} \vp(a_n) = +\infty$ and $|a_n| \to 1$ as $n \to +\infty$ and $|a_n-a_m|= \max( |a_m|,|a_n|)$ for all $m \neq n$, then $A$ is an analytic boundary.
\end{theo}


We finish with a simple result that allows us to construct more analytic boundaries.

\begin{lemm}
\label{compos}
If $A \subset D$ is an analytic boundary and $h(X) = \sum_{i \geq 1} h_i X^i \in X \cdot \OO_{\Cp} \dcroc{X}$ is such that $\inf_{i \geq 1} |h_i| = 0$ and $|h(a_n)| \to 1$ as $n \to + \infty$, then $h(A)$ is an analytic boundary.
\end{lemm}

\begin{proof}
The condition on $h(X)$ implies that $h$ gives rise to a surjective function $D \to D$ (if $y \in D$, consider the Newton polygon of $h(X)-y$).

Hence $\| f \|_D = \| f \circ h \|_D$. Now $\| f \circ h \|_D = \sup_{n \geq 1} |f \circ h(a_n)|$.
\end{proof}

\subsection{LT-like power series}
\label{polysub}

We define Lubin-Tate-like (LT-like) power series. Recall that the Weierstrass degree $\wideg(f)$ of $f(X) = \sum_{n \geq 0} f_n X^n \in \OO_{\Cp} \dcroc{X}$ is the min of the $n$ such that $f_n \in \OO_{\Cp}^\times$ (or $+\infty$ if there is no such $n$). Let $q$ be a power of $p$. 

\begin{defi}
\label{defgood}
An LT-like power series (of $\wideg$ $q$) is a power series $P(X) = \sum_{n \geq 1} p_n X^n \in \OO_{\Cp}\dcroc{X}$ with $0 < \vp(p_1) \leq 1$ and $p_q \in \OO_{\Cp}^\times$ and $P(X) \equiv p_q X^q \bmod{p_1}$.
\end{defi}

Note that if $P$ is a LT-like power series, then $P'(X)$ is a unit of $\calE^+_{\Cp}$. In particular, for every $z \in D$, all the roots of $P(X)-z$ in $D$ are simple. If $P(X)$ is a LT-like power series and $n \geq 0$, let $\Lambda_n = \{ z \in D$ such that $P^{\circ n}(z)=0 \}$, and let $\Lambda(P) = \cup_{n \geq 0} \Lambda_n$. 

The following theorem (theorem A) is proved at the end of \S \ref{auxsub}.

\begin{theo}
\label{lambdas}
If $P$ is LT-like, then $\Lambda(P)$ is an analytic boundary.
\end{theo}

\begin{rema}
\label{ltex}
Let $\pi$ be a uniformizer of a finite extension $F$ of $\Qp$ of degree $d$, and let $\LT$ be the Lubin-Tate formal $\OO_F$-module attached to $\pi$.
\begin{enumerate}
\item The power series $P(X)=[\pi](X)$ is an LT-like power series, and $\Lambda(P)=\LT[\pi^\infty]$ is therefore an analytic boundary by theorem \ref{lambdas}.
\item The $\Zp$-module $\LT[\pi^\infty]$ is isomorphic to $(\Qp/\Zp)^d$. If $M \subset  \LT[\pi^\infty]$ is isomorphic to $(\Qp/\Zp)^{d-1}$, then there is a nonzero bounded function $f(X) \in \calE_{\Cp}^+$ such that $f(z)=0$ for all $z \in M$. In particular, $M$ is not an analytic boundary.
\end{enumerate}
\end{rema}

Let $P$ be a LT-like power series, and write $P(X) = X \cdot Q(X)$. For $n \geq 1$, let $Q_n(X) = Q(P^{\circ (n-1)}(X))$, so that $P^{\circ n}(X) = X \cdot Q_1(X) \cdots Q_n(X)$. Let $q_n = q^{n-1}(q-1) = \wideg Q_n$ and $v_1 = \vp(p_1)$ and $\mu_n = v_1/q_n$. The $q_n$ roots of $Q_n$ are all of valuation $\mu_n$. 

Let $H_0=\{0\}$ and let $H_n$ be the set of roots of $Q_n$, so that $\Lambda_n = H_0 \sqcup H_1 \sqcup \hdots \sqcup H_n$.

\begin{lemm}
\label{srone}
Take $z,z' \in D$.
\begin{enumerate}
\item If $P(z) = P(z')$, then $\vp(z-z') \geq \mu_1$. 
\item If $P(z)=y$ and $P(z')=y'$ with $\vp(y-y') \geq \mu_n$, then $\vp(z-z') \geq \mu_{n+1}$.
\end{enumerate}
\end{lemm}

\begin{proof}
We prove both statements at the same time (for item (1), take $y=y'$). Recall that $P^{[i]}(X) = P^{(i)}(X)/i! \in \OO_{\Cp} \dcroc{X}$ is the $i$-th Hasse derivative. We have 
\[ P(X+z)-P(z') = (y-y') + P'(z) X + P^{[2]}(z) X^2 + \cdots + P^{[q]}(z) X^q + \bigO(X^{q+1}). \] 
The valuation of $P'(z)$ is $v_1$, the valuation of $P^{[q]}(z)$ is $0$, and the valuation of $P^{[i]}(z)$ is $\geq v_1$ for all $1 \leq i \leq q-1$. Indeed, $P^{[i]}(z) \equiv \binom{q}{i} p_q z^{q-i} \bmod{p_1}$ and since $q$ is a power of $p$, $\binom{q}{i}$ is divisible by $p$ for all $1 \leq i \leq q-1$ and hence by $p_1$. 

The lemma now follows from the theory of Newton polygons.
\end{proof}

\begin{coro}
\label{separoot}
If $k \geq 1$ and $P^{\circ k}(z) = P^{\circ k}(z')$, then $\vp(z-z') \geq \mu_k$.
\end{coro}

We now define a map $\psi$. Let $\phi : \calR^+_{\Cp} \to \calR^+_{\Cp}$ be the map defined by $\phi(f) = f \circ P$. Note that $\calR^+_{\Cp}$ is a free $\phi(\calR^+_{\Cp})$-module of rank $q$, generated for example by $1,X,\hdots,X^{q-1}$. Let $\psi : \calR^+_{\Cp} \to \calR^+_{\Cp}$ be the map defined by $\phi \circ \psi (f) = \Tr_{\calR^+_{\Cp} / \phi(\calR^+_{\Cp})} f$. Note that we have $\phi(\calE^+_{\Cp}) \subset \calE^+_{\Cp}$ and $\psi(\calE^+_{\Cp}) \subset \calE^+_{\Cp}$. Beware that in the literature, $\psi$ sometimes denotes the map that we have defined, but divided by $p_1$ or by $q$.

\begin{lemm}
\label{psidiv}
We have $\psi(\OO_{\Cp} \dcroc{X}) \subset p_1 \cdot \OO_{\Cp} \dcroc{X}$.
\end{lemm}

\begin{proof}
The $\phi(\OO_{\Cp} \dcroc{X})$-module $\OO_{\Cp} \dcroc{X}$ is free of rank $q$, generated by $1,X,\hdots,X^{q-1}$. A simple computation shows that mod $p_1$, the trace of $X^i$ is zero for $1 \leq i \leq q-1$. For $i=0$, it is $q$ which is divisible by $p_1$ since $P$ is LT-like.
\end{proof}

\begin{lemm}
\label{psisum}
If $f \in \calR^+$, then $\psi^n(f)(0) = \sum_{z \in \Lambda_n} f(z)$.
\end{lemm}

\begin{proof}
We have $\phi^n \circ \psi^n (f) = \Tr_{\calR^+ / \phi^n(\calR^+)} f$. If $I$ is the ideal of $\phi^n(\calR^+)$ generated by $P^{\circ n}(X)$, then $\phi^n(\calR^+) / I = \Cp$ and $\calR^+ / I = \calR^+ / P^{\circ n}(X) = \prod_{z \in \Lambda_n} \calR^+ / (X-z)$.
\end{proof}

\begin{prop}
\label{psibound}
If $f \in \calR^+$ then $\vp(\psi^n(f)(0)) \geq V(f,\mu_{n+1}) + (n-1) \cdot v_1$.
\end{prop}

\begin{proof}
Since $V(\sum_{i \geq 0} f_i X^i,\mu_{n+1}) = \inf_{i \geq 0} V(f_i X^i,\mu_{n+1})$, it is enough to prove the claim for $f(X) = X^i$. 
Write $Q_{n+1}(X) = \alpha_{n+1} (X^{q_{n+1}} +p_1 R_{n+1}(X))$ for some $\alpha_{n+1} \in \OO_{\Cp}^\times$ and $R_{n+1} \in \OO_{\Cp}\dcroc{X}$ and write $i=s q_{n+1}+r$ with $0 \leq r \leq q_{n+1}-1$. We have 
\[ X^i = X^{s q_{n+1}+r} = (\alpha_{n+1}^{-1} \cdot Q_{n+1}(X)- p_1 R_{n+1}(X))^s X^r = \sum_{k=0}^s Q_{n+1}(X)^k p_1^{s-k} F_k(X), \]
for some $F_k(X) \in \OO_{\Cp} \dcroc{X}$, $0 \leq k \leq s$. Since $Q_{n+1} = \phi^n(Q_1)$ and $\psi^n(F_k)(0) \in p_1^n \OO_{\Cp}$ by lemma \ref{psidiv}, we have $(\psi^n X^i)(0) \in p_1^{s+n} \OO_{\Cp}$. Hence \[ \vp(\psi^n (X^i)(0)) \geq s q_{n+1} \mu_{n+1} + n \cdot v_1 = i \mu_{n+1} -r \mu_{n+1} + n \cdot v_1 \geq V(X^i,\mu_{n+1}) + (n-1) \cdot v_1. \qedhere \]
\end{proof}

\subsection{Construction of auxilliary functions}
\label{auxsub}

The proof of theorem \ref{lambdas} rests on the construction of certain elements of $\calR^+$ satisying precise growth conditions. The proofs in this {\S} are very similar to those of Fresnel and de Mathan.

\begin{defi}
\label{deford}
We say that $f \in \calR^+$ is of $P$-order $1^-$ if $V(f,\mu_n) + n \cdot v_1 \to + \infty$ as $n \to + \infty$.
\end{defi}

\begin{rema}
\label{logrowth}
The infinite product $X \cdot \prod_{n \geq 1} Q_n(X)/p_1$ converges to a function $\log_P(X) \in \calR^+$ that satisfies: $\{ V(f,\mu_n) + n \cdot v_1 \}_{n \geq 1}$ is bounded below. Hence a function of $P$-order $1^-$ grows just slightly less fast than $\log_P(X)$.
\end{rema}

\begin{prop}
\label{ordpsi}
If $f$ is of $P$-order $1^-$, then $\sum_{z \in \Lambda_n} f(z) \to 0$ as $n \to +\infty$.
\end{prop}

\begin{proof}
This follows from lemma \ref{psisum} and prop \ref{psibound}.
\end{proof}

\begin{coro}
\label{nulsum}
If $f$ is of $P$-order $1^-$ and if $f(z) \to 0$ for $z \in \Lambda(P)$, then for all $i \geq 0$, we have $\sum_{z \in \Lambda(P)} z^i f(z) = 0$.
\end{coro}

\begin{proof}
If $i \geq 0$, then $X^i f(X)$ is also of $P$-order $1^-$. The result then follows from prop \ref{ordpsi} applied to $X^i f(X)$ since $\sum_{z \in \Lambda(P)} z^i f(z) = \lim_{n \to +\infty} \sum_{z \in \Lambda_n} z^i f(z)$.
\end{proof}

\begin{prop}
\label{bnfdm}
Take $n \geq 1$ and $0 \leq \lambda \leq 1$. 

There exists $B_n \subset H_n$ such that $\card B_n = \lfloor \lambda q_n \rfloor$, and such that for all $z \in H_n$ and $1 \leq k \leq n-1$, we have $\card \{ z' \in B_n$ such that $P^{\circ k}(z) = P^{\circ k}(z') \} \geq \lfloor \lambda q^k \rfloor$.
\end{prop}

\begin{proof}
For every $y \in H_{n-1}$, there are $q$ elements $z \in H_n$ such that $P(z)=y$. For each $y \in H_{n-1}$, choose $\lfloor \lambda q \rfloor$ of those $z$, and let $B^{(1)} \subset H_n$ denote all the $z$ chosen this way. Suppose that $2 \leq k \leq n-1$ and that we have constructed a set $B^{(k-1)} \subset H_n$. For each $y \in H_{n-k}$, there are $q^k$ elements $z \in H_n$ such that $P^{\circ k}(z)=y$. For each $y \in H_{n-k}$, choose $\lfloor \lambda q^k \rfloor$ of them, including all those of $B^{(k-1)}$. This is possible as $q \lfloor \lambda q^{k-1} \rfloor \leq \lfloor \lambda q^k \rfloor$. There are $q-1$ elements in $H_1$ so that $\card B^{(n-1)} = (q-1) \lfloor \lambda q^{n-1} \rfloor \leq  \lfloor \lambda q_n \rfloor$. We can now add some elements of $H_n$ to $B^{(n-1)}$ to get a set $B_n$ satisfying the conditions of the prop.
\end{proof}

Let $\lambda$ and $B_n$ be as in prop \ref{bnfdm} and let $D_n(X) = \prod_{\omega \in B_n} (1-X/\omega)$.

\begin{lemm}
\label{valdn}
For $n \geq 1$ and $z \in H_n \setminus B_n$, we have $\vp (D_n(z)) > (n-1) \lambda v_1 - \mu_1$.
\end{lemm}

\begin{proof}
Let $W_k = \{ z' \in B_n$ such that $P^{\circ k}(z) = P^{\circ k}(z') \}$ and let $w_k = \card W_k$. Note that $w_0=0$ since $z \notin B_n$. If $P^{\circ k}(z) = P^{\circ k}(z')$, then $\vp(z-z') \geq \mu_k$ by coro \ref{separoot}. Since $B_n = (W_1 \setminus W_0) \sqcup \hdots \sqcup (W_n \setminus W_{n-1})$, we have 
\begin{multline*}
\vp(D_n(z))  \geq \sum_{k = 1}^n (w_k-w_{k-1}) (\mu_k - \mu_n) 
 = \sum_{k=1}^{n-1} w_k (\mu_k - \mu_{k+1})  
> (n-1) \lambda v_1 - \mu_1,
\end{multline*}
since $w_k >  \lambda q^k  -1$ for $0 \leq k \leq n-1$ and $\mu_k-\mu_{k+1} = v_1/q^k$.
\end{proof}

\begin{theo}
\label{constfun}
For all $\eps > 0$ and $m \geq 1$, there exists $f_{\eps,m} \in \calR^+$ such that 
\begin{enumerate}
\item $f_{\eps,m}(0)=-1$ and $f_{\eps,m}(z)=0$ for all $z \in \Lambda_m \setminus \{0\}$;
\item $f_{\eps,m}$ is of $P$-order $1^-$;
\item $f_{\eps,m}(z) \to 0$ for $z \in \Lambda(P)$;
\item $\vp(f_{\eps,m}(z)) \geq -\eps$ for all $z \in \Lambda(P)$.
\end{enumerate}
\end{theo}

\begin{proof}
Let $\delta_1=\cdots=\delta_m=0$ and for $n \geq m+1$, take $\delta_n=q^{-\ell(n)}$ where $\ell(n)$ is the smallest integer $\geq 1$ such that $q^{-\ell(n)} \leq \eps/2n$. We assume that $\eps<1$, so that $\delta_n < 1$ for all $n$. We can also replace $m$ by a larger value, so that $\ell(n) \leq n-1$ for all $n \geq m+1$. In particular, $\lfloor \delta_n q_n \rfloor = \delta_n q_n$ for all $n$.

Let $\lambda_k = 1 - \delta_k$. Take $B_k$ as in prop \ref{bnfdm} with $\lambda=\lambda_k$ and let $D_k(X) = \prod_{\omega \in B_k} (1-X/\omega)$. Let $f \in \calR^+$ be $-1$ times the function provided by prop \ref{fdmzer} with $\eta=\eps/2$. Since $B_k=H_k$ for $1 \leq k \leq m$, this function satisfies (1). 

We have $V(D_k,\mu_n) = 0$ if $k \geq n$, so that $V(f,\mu_n) \geq \sum_{k = 1}^n V(D_k,\mu_n) - \eps/2$. Since $V(D_k,\mu_n) = b_k(\mu_n-\mu_k)$ where $b_k = \card B_k$, we have 
\[ V(f,\mu_n) +n \cdot v_1 \geq v_1 \cdot \sum_{k=1}^n \delta_k + (b_1+\cdots+b_n)\mu_n-\eps/2 > v_1 \cdot \sum_{k=1}^n \delta_k -\eps/2. \]
Since $\sum_{k=1}^n \delta_k \to + \infty$ as $n \to + \infty$, $f$ satisfies (2). 
Write $f(X) = D_n(X) f_n(X)$. If $z \in B_n$, then $f(z)=0$, while if $z \in H_n \setminus B_n$, then $\vp(f(z)) = \vp(f_n(z)) + \vp(D_n(z))$, and $\vp(f_n(z)) \geq V(f_n,\mu_n) = V(f,\mu_n)$ since $V(D_n,\mu_n)=0$. We have $b_k = (1-\delta_k) q_k$, so 
\begin{multline*} 
V(f,\mu_n) \geq  \sum_{k = 1}^{n-1} (1-\delta_k) q_k (\mu_n-\mu_k) - \eps/2  \\ \geq \mu_1 - \mu_n - (n-1) v_1 + \sum_{k=1}^{n-1} \delta_k (v_1 - q_k \mu_n) - \eps/2
\end{multline*}

By lemma \ref{valdn}, we have $\vp(D_n(z)) \geq (n-1) (1-\delta_n) v_1 - \mu_1$, so that
\[ \vp(f(z)) \geq - \mu_n + \delta_n v_1 - n \delta_n v_1 +  \sum_{k=1}^{n-1} \delta_k (v_1 - q_k \mu_n)- \eps/2. \]
We have $v_1 - q_k \mu_n \geq v_1 \cdot (1-1/q)$ and $n \delta_n v_1 \leq \eps/2$ and $- \mu_n + \delta_n v_1 \geq 0$ and $\sum_{k=1}^n \delta_k \to + \infty$ as $n \to + \infty$, so that $f$ satisfies (3) and (4).
\end{proof}

We can now prove theorem \ref{lambdas}.

\begin{proof}[Proof of theorem \ref{lambdas}]
Let $\Lambda'_m = \Lambda(P) \setminus \Lambda_m$. We prove that $\Lambda'_m$ is an analytic boundary for all $m \geq 1$. By coro \ref{nulsum}, the function provided by theorem \ref{constfun} has the property that $\sum_{z \in \Lambda'_m} z^i f_{\eps,m}(z) = 0$ for all $i \geq 1$ and $\sum_{z \in \Lambda'_m} f_{\eps,m}(z) = 1$. 

Take $h(X)  = \sum_{i \geq 0} h_i X^i \in \calE^+_{\Cp}$. We have \[ \sum_{z \in \Lambda'_m} f_{\eps,m}(z) h(z) = \sum_{i \geq 0} h_i \sum_{z \in \Lambda'_m}  f_{\eps,m}(z) z^i = h_0. \]
Hence $\vp(h_0) \geq \inf_{z \in \Lambda'_m} \vp(h(z)) - \eps$. This holds for all $\eps>0$, so that $\vp(h_0) \geq \inf_{z \in \Lambda'_m} \vp(h(z))$. 

Applying the same reasoning to $(h(X)-h_0)/X$ and to $m_1 \geq m$ gives us 
\[ \vp(h_1) \geq \inf_{z \in \Lambda'_{m_1}} \vp(h(z)) - \mu_{m_1} \geq \inf_{z \in \Lambda'_m} \vp(h(z)) - \mu_{m_1}. \]
This holds for all $m_1 \geq m$, so that $\vp(h_1) \geq \inf_{z \in \Lambda'_m} \vp(h(z))$. We repeat this, and we get that $\vp(h_i) \geq \inf_{z \in \Lambda'_m} \vp(h(z))$ for all $i \geq 0$, so that $\| h \|_D \leq \sup_{z \in \Lambda'_m} |h(z)|$.
\end{proof}

\section{Applications to $p$-adic Fourier theory}
\label{fousec}

In this section, we give an application of the surjectivity of the $p$-adic Fourier transform to $p$-adic Fourier theory and the geometry of the character variety. 

\subsection{$p$-adic Fourier theory}
\label{subfou}

Let $F$ be a finite extension of $\Qp$ of degree $d$, with ring of integers $\OO_F$. We first extend the Fourier transform to $\OO_F$. Let $X_{\tor}$ denote the set of finite order characters $(\OO_F,+) \to (\Cp^\times,\times)$. Given $z \in c^0(X_{\tor},\Cp)$, its Fourier transform is the function $\calF(z) : \OO_F \to \Cp$ defined by $a \mapsto \sum_{g \in X_{\tor}} z_g \cdot g(a)$.

\begin{theo}
\label{fdmof}
The map $\calF : c^0(X_{\tor},\Cp) \to C^0(\OO_F,\Cp)$ is surjective.
\end{theo}

\begin{proof}
If we choose a basis $a_1,\hdots,a_d$ of $\OO_F$ over $\Zp$, then there are linear forms $c_1,\hdots,c_d : \OO_F \to \Zp$ (the dual basis of the $a_i$'s) such that every $a \in \OO_F$ can be written as $a= \sum_{i=1}^d c_i(a) \cdot a_i$. Every finite order character $\OO_F \to \Cp^\times$ is then of the form $a \mapsto \gamma_1^{c_1(a)} \cdots \gamma_d^{c_d(a)}$ with $\gamma_1,\hdots,\gamma_d \in \Gamma$. We therefore have 
\[ c^0(X_{\tor},\Cp) = c^0(\Gamma,\Cp) \hotimes \cdots \hotimes c^0(\Gamma,\Cp). \]
Likewise, the decomposition $\OO_F = \Zp \cdot a_1 \oplus \cdots \oplus \Zp \cdot a_d$ gives us an isomorphism \[ C^0(\OO_F,\Cp) = C^0(\Zp,\Cp) \hotimes \cdots \hotimes C^0(\Zp,\Cp). \]
The theorem now follows from the surjectivity (see \cite{FdM74,FdM75,FdM78}) of the Fourier transform $c^0(\Gamma,\Cp) \to C^0(\Zp,\Cp)$.
\end{proof}

We now turn to $p$-adic Fourier theory. Let $e$ be the ramification index of $F$, let $\pi$ be a uniformizer of $\OO_F$, and let $q = \card \OO_F/\pi$. Let $\LT$ be the Lubin-Tate formal $\OO_F$-module attached to $\pi$, let $X$ be a coordinate on $\LT$ and let $\log_{\LT}(X)$ be the logarithm of $\LT$.
Recall (see \S 3 and \S 4 of \cite{ST01} for what follows) that $\Hom_{\OO_{\Cp}} (\LT,\Gm) \neq \{ 0 \}$. Choosing a generator of this group gives a power series $G(X) \in X \cdot \OO_{\Cp}\dcroc{X}$ such that $G(X) = \Omega \cdot X + \cdots$, where $\Omega \in \OO_{\Cp}$ with $\vp(\Omega) = 1/(p-1)-1/e(q-1)$. In particular, $1 + G(X) = \exp(\Omega \cdot \log_{\LT}(X)) = \sum_{n \geq 0} P_n(\Omega) X^n$ where $P_n(Y) \in F[Y]$ is a polynomial of degree $n$ such that $P_n(\Omega \cdot \OO_F) \subset \OO_{\Cp}$. 

When $F=\Qp$ and $\LT=\Gm$, we have $\Omega=1$ and $P_n(Y) = \binom{Y}{n}$. The family $\{ \binom{Y}{n} \}_{n \geq 0}$ forms a Mahler basis of $\Zp$. In addition, by a theorem of Amice (see \cite{A64}), every locally analytic function $\Zp \to \Cp$ can be written as $x \mapsto \sum_{n \geq 0} c_n \binom{x}{n}$ where $\{c_n\}_{n \geq 0}$ is a sequence of $\Cp$ such that there exists $r>1$ satisfying $|c_n| \cdot r^n \to 0$. 

One of the main results of $p$-adic Fourier theory is the following generalization of Amice's theorem (prop 4.5 and theo 4.7 of \cite{ST01}).

\begin{theo}
\label{mahlocan}
If $\{c_m\}_{m \geq 0}$ is a sequence of $\Cp$ such that there exists $r>1$ satisfying $|c_m| \cdot r^m \to 0$, then $a \mapsto \sum_{m \geq 0} c_m P_m(a \Omega)$ is a locally $F$-analytic function $\OO_F \to \Cp$.

Conversely, every locally $F$-analytic function $\OO_F \to \Cp$ has a unique such expansion.
\end{theo}

If $z \in D$, then (see \S 3 of \cite{ST01}) the map $\kappa_z : \OO_F \to \Cp$ given by 
\[ \kappa_z(a) = 1+G([a](z)) = \sum_{n \geq 0} P_n(a \Omega) z^n \] 
is a locally $F$-analytic character $(\OO_F,+) \to (\Cp^\times,\times)$, and every such character is of this form for a unique $z \in D$. Furthermore, $\kappa_z$ is of finite order if and only if $z \in \LT[\pi^\infty]$ (hence the set $X_{\tor}$ of torsion characters $(\OO_F,+) \to (\Cp^\times,\times)$ corresponds to $\LT[\pi^\infty]$).

\begin{defi}
\label{deftlt}
Let $\calF : c^0(\LT[\pi^\infty],\Cp) \to C^0(\OO_F,\Cp)$ be the map given by $\calF(\lambda)(a) = \sum_{\omega \in \LT[\pi^\infty]} \lambda_\omega \cdot \kappa_\omega(a)$.
\end{defi}

\begin{prop}
\label{fdmlt}
The map $\calF : c^0(\LT[\pi^\infty],\Cp) \to C^0(\OO_F,\Cp)$ is surjective.
\end{prop}

\begin{proof}
Since $X_{\tor} = \{ \kappa_\omega, \omega \in \LT[\pi^\infty] \}$, this follows from theorem \ref{fdmof}.
\end{proof}

\begin{theo}
\label{mahlsurj}
The map $c^0(\NN, \Cp) \to C^0(\OO_F,\Cp)$ given by $c \mapsto \sum_{m \geq 0} c_m P_m(\cdot \Omega)$ is surjective.
\end{theo}

\begin{proof}
Take $f \in C^0(\OO_F,\Cp)$. By prop \ref{fdmlt}, we can write $f = \sum_{\omega \in \LT[\pi^\infty]} \lambda_\omega \kappa_\omega$. We have $\kappa_\omega(a) = \sum_{n \geq 0} P_n(a \Omega) \omega^n$. This implies the corollary, with $c_m = \sum_{\omega \in \LT[\pi^\infty]} \lambda_\omega \omega^m$.
\end{proof}

\begin{prop}
\label{noninj}
If $F \neq \Qp$, the map $c^0(\NN, \Cp) \to C^0(\OO_F,\Cp)$ is not injective.
\end{prop}

\begin{proof}
If the map was injective, it would be a topological isomorphism by the open mapping theorem. For $n \geq 0$, we have
\[ \Id_{\pi^n \OO_F}(a) = q^{-n} \cdot \sum_{[\pi^n](\omega)=0} \kappa_{\omega}(a) = \sum_{k \geq 0} c_{k,n} P_k(a \Omega) \] with $c_{k,n} = q^{-n} \sum_{[\pi^n](\omega)=0} \omega^k$ (and no other choice if the map is injective). 

Take $P(X) = [\pi](X)$ and let $\psi$ be as in \S \ref{polysub}. By lemma \ref{psisum}, we have $c_{k,n} =  q^{-n} \cdot \psi^n(X^k)(0)$. By lemma \ref{supsi} below, we have $\sup_{k \geq 0} |\psi^n(X^k)(0)| = |\pi^n|$, so that $\sup_{k \geq 0} |c_{k,n}| = |(\pi/q)^n|$ is unbounded as $n \to + \infty$ if $\vp(q) > \vp(\pi)$. 
\end{proof}

\begin{lemm}
\label{supsi}
We have $\sup_{k \geq 0} |\psi^n(X^k)(0)| = |\pi^n|$
\end{lemm}

\begin{proof}
Since $\psi(\OO_{\Cp} \dcroc{X}) \subset \pi \cdot \OO_{\Cp} \dcroc{X}$ by lemma \ref{psidiv}, we have one inequality. Conversely, $\psi^n(f \circ P^{\circ (n-1)}) = \pi^{n-1} \psi(f)$, and if $f(X)=P(X)/X$, then $\psi(f)(0)=\pi$.
\end{proof}

\subsection{The Peano map}
\label{peano}

Recall that every element of $C^0(\Zp,\Cp)$ can be written in one and only one way as $x \mapsto \sum_{n \geq 0} \lambda_n \binom{x}{n}$ where $\lambda_n \in \Cp$ and $\lambda_n \to 0$. 

Let $T : C^0(\Zp,\Cp) \to C^0(\OO_F,\Cp)$ be the map given by \[ T : \left[ x \mapsto \sum_{n \geq 0} \lambda_n \binom{x}{n} \right] \mapsto \left[ a \mapsto \sum_{n \geq 0} \lambda_n P_n( a \Omega) \right]. \]

We can now prove theorem B.

\begin{coro}
\label{thb}
The map $T : C^0(\Zp,\Cp) \to C^0(\OO_F,\Cp)$ is surjective, and noninjective if $F \neq \Qp$.
\end{coro}

\begin{proof}
This follows from theorem \ref{mahlsurj} and prop \ref{noninj}.
\end{proof}

We identify the dual of the $\Cp$-Banach space $C^0(\Zp,\Cp)$ with $\calE^+_{\Cp}$ via the Amice transform. Let $\Lambda(\OO_F)$ denote the space $\Cp \otimes_{\OO_{\Cp} }\OO_{\Cp} \dcroc{\OO_F}$ of $\Cp$-valued measures on $\OO_F$, so that $\Lambda(\OO_F)$ is the dual of $C^0(\OO_F,\Cp)$ (and note that $\Lambda(\Zp) \simeq\calE^+_{\Cp}$). If $a_1,\hdots,a_d$ is a basis of $\OO_F$ over $\Zp$, the ring $\Lambda(\OO_F)$ is isomorphic to $\Cp \otimes_{\OO_{\Cp}}\OO_{\Cp} \dcroc{X_1,\hdots,X_d}$ where $X_i = \delta_{a_i} - \delta_0$ (note that $\delta_0=1$). There is an algebra homomorphism $\Lambda(\OO_F) \to \calE^+_{\Cp}$ that sends $\delta_b-1$ to $G([b](X))$, and by lemma 1.15 of \cite{BSX}, this map is injective. 

\begin{prop}
\label{peandual}
The dual map $T' : \Lambda(\OO_F) \to \calE^+_{\Cp}$ is the above inclusion.
\end{prop}

\begin{proof}
Take $b \in \OO_F$. We have $T'(\delta_b) ( x \mapsto \binom{x}{n} ) = \delta_b ( a \mapsto P_n( a \Omega)) = P_n(b \Omega)$ so that the image of $\delta_b$ in $\calE^+_{\Cp}$ is $\sum_{n \geq 0} P_n(b \Omega) X^n = 1+G([b](X))$.
\end{proof}

\begin{prop}
\label{imsprcl}
The image of $T' : \Lambda(\OO_F) \to \calE^+_{\Cp}$ is closed in $\calE^+_{\Cp}$.
\end{prop}

\begin{proof}
Since $T$ is surjective, and $C^0(\OO_F,\Cp)$ is a $\Cp$-Banach space of countable type, $T'$ has closed image by prop \ref{clrgpinv} below (the closed range theorem).
\end{proof}

\begin{prop}
\label{clrgpinv}
If $T : X \to Y$ is a continuous map of $\Cp$-Banach spaces, and if $Y$ is of countable type and $\im(T)$ is closed in $Y$, then $\im(T')$ is closed in $X'$.
\end{prop}

\begin{proof}
The result follows from theorem 3.1, (ii) and (i), of \cite{HNA}, given the remarks on page 202 of ibid.
\end{proof}

\begin{coro}
\label{injifcl}
The map $T' : \Lambda(\OO_F) \to \calE^+_{\Cp}$ is an isometry on its image.
\end{coro}

\begin{proof}
The map $T'$ is injective, it is an algebra homomorphism, and $\| T'(f) \| \leq \| f \|$. If $T'$ is not an isometry, there is some $f \in \Lambda(\OO_F)$ such that $\| T'(f) \| = C \cdot \| f \|$ with $C < 1$. We then have $\| T'(f^n) \| \leq C^n \cdot \| f^n \|$. This contradicts the continuity of the map $(T')^{-1} : \im(T') \to \Lambda(\OO_F)$ provided by prop \ref{imsprcl} and the open mapping theorem.
\end{proof}

Note that $T'$ is not surjective if $F \neq \Qp$ as $T$ is not injective by prop \ref{noninj}. Indeed, $\ker(T) = {}^\perp \im(T')$ as the dual of a space of countable type separates its points.

\subsection{The character variety}
\label{cvsub}

Schneider and Teitelbaum have constructed in \cite{ST01} a $1$-dimensional rigid analytic group variety $\frX \subset \frB^d$ over $F$, called the character variety, whose closed points in an extension $K/F$ parameterize locally $F$-analytic characters $\OO_F \to K^\times$. They show that over $\Cp$, the variety $\frX$ becomes isomorphic to $\frB$. On the level of points, the isomorphism $\frB \to \frX$ is given by the map $z \mapsto \kappa_z$ recalled in \S \ref{subfou}.

The ring $\OO^b_{\Cp}(\frB^d)$ of bounded functions on $\frB^d$ defined over $\Cp$ is isomorphic to $\Lambda(\OO_F)$ and the ring $\OO^b_{\Cp}(\frX)$ of bounded functions on $\frX$ defined over $\Cp$ is isomorphic to $\calE^+_{\Cp}$. The restriction-to-$\frX$-map $\res_{\frX} : \OO^b_{\Cp}(\frB^d) \to \OO^b_{\Cp}(\frX)$ then corresponds to the inclusion $T' : \Lambda(\OO_F) \to \calE^+_{\Cp}$ considered in \S \ref{peano}. In particular, coro \ref{injifcl} implies the following result, which is theorem C.

\begin{theo}
\label{resisom}
The map $\res_{\frX} : \OO^b_{\Cp}(\frB^d) \to \OO^b_{\Cp}(\frX)$ is an isometry on its image.
\end{theo}

It is possible to characterize the image of $\res_{\frX}$, see prop 3.1.8 of \cite{AB24} for a proof of the following result. 

\begin{prop}
\label{imocpof}
The image of $\res_{\frX}$ is the set of power series $f(X) \in \calE^+_{\Cp}$ such that $\{ q^{-n} \cdot \psi^n( G([a](X)) \cdot f(X) ) \}_{a,n}$ is bounded in $\calE^+_{\Cp}$ as $a \in o_F$ and $n \geq 0$.
\end{prop}

We finish by stating and proving theorem D (the only result of this section on $p$-adic Fourier theory that uses theorem A beyond the cyclotomic case).

\begin{theo}
\label{suptors}
If $f \in \OO^b_{\Cp}(\frX)$, then $\| f \|_{\frX} = \sup_{\kappa \in X_{\tor}} |f(\kappa)|$. 
\end{theo}

\begin{proof}
In the isomorphism between $\frX$ and $\frB$, the set $X_{\tor}$ of torsion characters $(\OO_F,+) \to (\Cp^\times,\times)$ corresponds to $\LT[\pi^\infty]$, and $\OO^b_{\Cp}(\frX)$ is isomorphic to $\calE^+_{\Cp}$. 

Theorem A applied to $P(X)=[\pi](X)$ then implies the result.
\end{proof}

\providecommand{\og}{``}
\providecommand{\fg}{''}


\begin{thebibliography}{HNA05}

\bibitem[Ami64]{A64}
{\scshape Y.~Amice} -- {\og Interpolation {$p$}-adique\fg}, \emph{Bull. Soc.
  Math. France} \textbf{92} (1964), p.~117--180.

\bibitem[AB24]{AB24}
{\scshape K.~Ardakov {\normalfont \&} L.~Berger} -- {\og Bounded
  functions on the character variety\fg}, \emph{M\"{u}nster J. Math.},
  to appear.

\bibitem[BSX20]{BSX}
{\scshape L.~Berger, P.~Schneider {\normalfont \&} B.~Xie} -- {\og
  Rigid character groups, {L}ubin-{T}ate theory, and
  {$(\varphi,\Gamma)$}-modules\fg}, \emph{Mem. Amer. Math. Soc.} \textbf{263}
  (2020), no.~1275, p.~v+79.

\bibitem[Bou10]{B10}
{\scshape K.~Boussaf} -- {\og Identity theorem for bounded {$p$}-adic
  meromorphic functions\fg}, \emph{Bull. Sci. Math.} \textbf{134} (2010),
  no.~1, p.~44--53.

\bibitem[Esc95]{E95}
{\scshape A.~Escassut} -- \emph{Analytic elements in {$p$}-adic analysis},
  World Scientific Publishing Co., Inc., River Edge, NJ, 1995.

\bibitem[FdM74]{FdM74}
{\scshape J.~Fresnel {\normalfont \&} B.~de~Mathan} -- {\og L'image de
  la transformation de {F}ourier {$p$}-adique\fg}, \emph{C. R. Acad. Sci. Paris
  S\'{e}r. A} \textbf{278} (1974), p.~653--656.

\bibitem[FdM75]{FdM75}
{\scshape J.~Fresnel {\normalfont \&} B.~de~Mathan} -- {\og Transformation de {F}ourier {$p$}-adique\fg}, in
  \emph{Journ\'{e}es {A}rithm\'{e}tiques de {B}ordeaux ({C}onf., {U}niv.
  {B}ordeaux, {B}ordeaux, 1974)}, Ast\'{e}risque, No. 24-25, Soc. Math. France,
  Paris, 1975, p.~139--155.

\bibitem[FdM78]{FdM78}
{\scshape J.~Fresnel {\normalfont \&} B.~de~Mathan} -- {\og Alg\`ebres {$L\sp{1}$} {$p$}-adiques\fg}, \emph{Bull. Soc. Math.
  France} \textbf{106} (1978), no.~3, p.~225--260.

\bibitem[HNA05]{HNA}
{\scshape H.~R. Henr\'{\i}quez, S.~Navarro {\normalfont \&}
  J.~A{guayo-Garrido}} -- {\og Closed linear operators between nonarchimedean
  {B}anach spaces\fg}, \emph{Indag. Math. (N.S.)} \textbf{16} (2005), no.~2,
  p.~201--214.

\bibitem[Laz62]{L62}
{\scshape M.~Lazard} -- {\og Les z\'{e}ros des fonctions analytiques d'une
  variable sur un corps valu\'{e} complet\fg}, \emph{Inst. Hautes \'{E}tudes
  Sci. Publ. Math.} (1962), no.~14, p.~47--75.

\bibitem[Rob00]{R00}
{\scshape A.~M. Robert} -- \emph{A course in {$p$}-adic analysis}, Graduate
  Texts in Mathematics, vol. 198, Springer-Verlag, New York, 2000.

\bibitem[ST01]{ST01}
{\scshape P.~Schneider {\normalfont \&} J.~Teitelbaum} -- {\og
  {$p$}-adic {F}ourier theory\fg}, \emph{Doc. Math.} \textbf{6} (2001),
  p.~447--481.

\end{thebibliography}
\end{document}